\documentclass[12pt]{article}
\usepackage{fullpage}
\usepackage{mathrsfs,pstricks,ifpdf,tikz}
\usetikzlibrary{calc}
\usetikzlibrary{arrows}
\usepackage{amsfonts}
\usepackage{enumerate}

\usepackage{makecell}

\usepackage{latexsym}
\usepackage{amsmath,amssymb}
\usepackage{amsthm}
\usepackage{ifthen,calc}
\usepackage{color}
\usepackage{graphicx}
\usepackage{latexsym}
\usepackage{multirow}
\usepackage{diagbox}
\usepackage{float}
\usepackage[format=hang, margin=10pt]{caption}

\newtheorem{theorem}{Theorem}[section]

\newtheorem{corollary}[theorem]{Corollary}
\newtheorem{conjecture}[theorem]{Conjecture}
\newtheorem{lemma}[theorem]{Lemma}

\newtheorem{problem}[theorem]{Problem}

\theoremstyle{definition}
\newtheorem{definition}[theorem]{Definition}

\newtheorem{observation}[theorem]{Observation}

\def\VEC#1#2#3{{#1}_{#2},\ldots,{#1}_{#3}}
\def\st{\colon\,}
\def\la{\langle}
\def\ra{\rangle}
\def\CL#1{\left\lceil{#1}\right\rceil}
\def\FL#1{\left\lfloor{#1}\right\rfloor}
\def\CH#1#2{\left({#1 \atop #2}\right)}
\def\hG{\hat G}
\def\tK{\widetilde K}
\def\C#1{\left\vert #1\right\vert}
\def\esub{\subseteq}

\begin{document}

\title{\bf\Large Some new results on bar visibility of digraphs}

\author{Yuanrui Feng\thanks{School of Mathematics, Tianjin University, Tianjin, China.}\,
\qquad
Jun Ge\thanks{School of Mathematical Sciences, Sichuan Normal University, Chengdu, China: mathsgejun@163.com.
Supported by NNSF of China under Grant NSFC-11701401.} \qquad
Douglas B. West\thanks{Zhejiang Normal University, Jinhua, China, and
University of Illinois, Urbana, USA: dwest@math.uiuc.edu.
Supported by NNSF of China under Grant NSFC-11871439.}
\qquad
Yan Yang\thanks{School of Mathematics, Tianjin University, Tianjin, China: yanyang@tju.edu.cn. Supported by NNSF of China under Grant NSFC-11971346. } }

\date{\today}

\maketitle

\begin{abstract}
Visibility representation of digraphs was introduced by Axenovich, Beveridge,
Hutch\-inson, and West (\emph{SIAM J. Discrete Math.} {\bf 27}(3) (2013)
1429--1449) as a natural generalization of $t$-bar visibility representation of
undirected graphs.  A {\it $t$-bar visibility representation} of a digraph $G$
assigns each vertex at most $t$ horizontal bars in the plane so that there is
an arc $xy$ in the digraph if and only if some bar for $x$ ``sees'' some bar
for $y$ above it along an unblocked vertical strip with positive width.
The {\it visibility number} $b(G)$ is the least $t$ such that $G$ has a $t$-bar
visibility representation.  In this paper, we solve several problems about
$b(G)$ posed by Axenovich et al.\ and prove that determining whether the bar
visibility number of a digraph is $2$ is NP-complete.

\medskip

\noindent {\bf Keywords:} bar visibility number, graph representation, transitive tournament, NP-complete.

\smallskip
\noindent {\bf Mathematics Subject Classification (2010): 05C62, 05C35, 05C10,
68Q17}
\end{abstract}

\section{Introduction}

Visibility representation of graphs has been studied extensively in
computational geometry and has important application in VLSI design, computer
vision, etc.; for a book devoted to the topic, see Ghosh~\cite{Gho}.
Among various types of visibility representations of graphs, we focus here
on bar visibility representation in the plane.

A graph $H$ is a {\it bar visibility graph} if each vertex can be assigned a
horizontal line segment in the plane (called a {\it bar}) so that vertices are
adjacent if and only if the corresponding bars can see each other along an
unblocked channel, where a {\it channel} is a vertical strip of positive
width.  The assignment of bars is a {\it bar visibility representation} of $H$.
Tamassia and Tollis~\cite{Tamassia86} and Wismath~\cite{Wismath85} characterized
bar visibility graphs (see Hutchinson~\cite{Hutch} for another proof).

\begin{theorem}[\cite{Tamassia86,Wismath85}]\label{barvG}
A graph $H$ has a bar visibility representation if and only if $H$ can be
embedded in the plane so that all cut-vertices appear on the boundary of one
face.
\end{theorem}

Chang, Hutchinson, Jacobson, Lehel, and West~\cite{Chang04} extended this
concept to all graphs by introducing $t$-bar visibility representations of
graphs.  A {\it $t$-bar visibility representation} of a graph $H$ assigns each
vertex up to $t$ bars in the plane so that two vertices are adjacent if and
only if some bar for one vertex can see some bar for the other via an
unblocked channel.  The least $t$ such that $H$
has a $t$-bar visibility representation is called the {\it bar visibility
number} of $H$, denoted by $b(H)$.

Axenovich, Beveridge, Hutchinson, and West \cite{ABHW2013} introduced an
analogue for directed graphs.  A {\it $t$-bar visibility representation} of a
digraph $G$ assigns each vertex at most $t$ bars in the plane so that there is
an arc $xy$ in the digraph if and only if some bar for $x$ sees some bar for
$y$ above it via an unblocked channel.  The {\it
bar visibility number} $b(G)$ of a digraph $G$ is the least $t$ such that $G$
has a $t$-bar visibility representation.  Digraphs with bar visibility number
$1$ are {\it bar visibility digraphs}.

In a digraph, a vertex is a {\it source} or a {\it sink} if it has indegree $0$
or outdegree $0$, respectively. A {\it consistent cycle} is an oriented cycle
with no source or sink.  Tomassia and Tollis~\cite{Tamassia86} and
independently Wismath~\cite{Wismath89} characterized bar visibility digraphs.

\begin{theorem}[\cite{Tamassia86, Wismath89}]\label{barvdig}
Let $G$ be a digraph, and let $G'$ be the digraph formed from $G$ by adding two
vertices $s$ and $t$, an arc $sv$ for every source vertex $v$ in $G$, an arc
$wt$ for every sink vertex $w$, and the arc $st$. A digraph $G$ is a bar
visibility digraph if and only if $G'$ is planar and has no consistent cycle.
\end{theorem}

Thus planarity is necessary but not sufficient for $b(G)=1$.
Axenovich, Beveridge, Hutchinson, and West~\cite{ABHW2013} showed that
$b(G)\leq 4$ when $G$ is a planar digraph, $b(G)\leq 2$ when $G$ is
outerplanar, and in general $b(G)\leq (\C{V(G)}+10)/3$.  For outerplanar
digraphs, West and Wise \cite{West17} gave a forbidden substructure
characterization for those with $b(G)=1$.

A {\it tournament} is an orientation of a complete graph.  A tournament $T$ is
{\it transitive} if $xz$ is an arc whenever $xy$ and $yz$ are arcs.  In
particular, $T$ is transitive if and only if there is a linear ordering of the
vertices such that $xy$ is an arc if and only if $x$ precedes $y$ in the
ordering.  Up to isomorphism, there is only one transitive tournament on $n$
vertices, denoted by $T_n$.  In \cite{ABHW2013}, the authors gave the exact
value of $b(T_n)$ for $1\leq n\leq 15$ except for $n\in\{11,12\}$, and they
gave two upper bounds for $b(T_n)$ by using Steiner systems.

\begin{theorem}[\cite{ABHW2013}]\label{smalln}
The bar visibility number of the transitive tournament $T_n$ satisfies
\begin{displaymath}
b(T_n) = \left\{ \begin{array}{ll}
1, & \text{if ~$1\leq n\leq 4$}, \\
2, & \text{if ~$5\leq n\leq 10$}, \\
3, & \text{if ~$13\leq n\leq 15$.}
\end{array} \right.
\end{displaymath}
\end{theorem}

\begin{theorem}[\cite{ABHW2013}]\label{upper}
The bar visibility number of the transitive tournament $T_n$ satisfies\\
(1) $b(T_n)\leq \frac{7n}{24}+2\sqrt{n\log n}$;\\
(2) $b(T_n)< \frac{3n}{14}+42$ when n is sufficiently large.
\end{theorem}

Axenovich et al.~\cite{ABHW2013} posed two open problems and two conjectures
that we address here.

\begin{problem}[\cite{ABHW2013}]\label{problem1}
What is the least $\alpha$ such that always $b(T_n)\leq \alpha n+c$ for some
fixed $c$?
\end{problem}

\begin{problem}[\cite{ABHW2013}]\label{problem2}
What is $\lim\limits_{n \to \infty} {b(T_n)}/{n}$ (if the limit exists)?
\end{problem}

\begin{conjecture}[\cite{ABHW2013}]\label{con1}
$b(T_{11})=3$.
\end{conjecture}

\begin{conjecture}[\cite{ABHW2013}]\label{con2}
If $G$ is an orientation of an undirected graph $\hG$, then $b(G)\leq 2b(\hG)$.
\end{conjecture}

In Section 2, we present a simple construction proving $b(T_n)\le\CL{n/4}$.
This does not improve the upper bound when $n$ is sufficiently large but is
valid for all $n$, improving statement (1) of Theorem~\ref{upper}.  In Section
3, we prove that $\lim\limits_{n \to \infty} {b(T_n)}/{n}$ exists and is at
least $(3-\sqrt 7)/2$, about $0.177124$.  This improves the easy lower bound of
$1/6$, mentioned in~\cite{ABHW2013}, that follows from Euler's Formula.
As a consequence of our lower bound, we prove Conjecture~\ref{con1}; in
particular, $b(T_{11})=b(T_{12})=3$ and $b(T_{17})=4$.
In Section 4, we disprove Conjecture~\ref{con2} for $b(\hG)=1$ but in general
observe $b(G)\le4b(\hG)$.  Finally, in Section 5 we prove that determining
whether $b(G)\leq 2$ is NP-complete.

A simple observation is helpful in studying $b(T_n)$ for small $n$.

\begin{lemma}\label{n+2}
$b(T_{n})\leq b(T_{n+1})\leq b(T_{n+2})\leq b(T_{n})+1$.
\end{lemma}
\begin{proof}
Because $T_n$ is transitive, removing bars from a visibility representation of
$T_n$ cannot add any unwanted visibility.  Thus we can obtain an $m$-bar
visibility representation of $T_{n}$ from one for $T_{n+1}$ by removing the
bars for one vertex, and similarly $b(T_{n+1})\le b(T_{n+2})$.

To complete the proof, we obtain a $(k+1)$-bar visibility representation of
$T_{n+2}$ from a $k$-bar visibility representation of $T_n$.  Draw the
representation of the smaller tournament with vertices $\VEC v1n$ in the left
half-plane.  In the right half-plane, we will add one bar for each of
$\VEC v0{n+1}$, representing all arcs involving the two new vertices
$v_0$ and $v_{n+1}$ at a cost of adding one new bar for each old vertex.

Index the vertices so that $v_0$ is a source and $v_{n+1}$ is a sink in
$T_{n+2}$, making $T_{n+2}$ indeed transitive.  For $1\le i\le n$,
assign to $v_i$ the bar from the point $(i-1,1)$ to the point $(i,1)$.
Assign to $v_0$ the bar from $(0,0)$ to $(n+1,0)$, and assign
$v_{n+1}$ the bar from $(0,2)$ to $(n+1,2)$.  This generates arcs from
$v_0$ to all of $\VEC v1{n+1}$ and from all of $\VEC v0n$ to $v_{n+1}$.
\end{proof}

\section{An upper bound on $b(T_n)$}

In this section, we prove an upper bound on $b(T_n)$ for general $n$ by
using decompositions of the complete graph.  A well-known result about complete
graphs of even order is that they decompose into spanning paths.

\begin{lemma}[\cite{Alspach08}]\label{h}
The complete graph $K_{2m}$ with vertex set $\{x_1, \ldots, x_{2m}\}$
decomposes into spanning paths $P_1,\ldots, P_{m}$ given by
\begin{equation}\label{p}
P_{i}=\la x_ix_{i+1}x_{i-1}x_{i+2}x_{i-2}\cdots x_{i+(m-1)}x_{i-(m-1)}x_{i+m}\ra
\end{equation}
for $1\leq i\leq m$, with subscripts on $x$ taken modulo $2m$.
\end{lemma}
For $1\leq i\leq m$, the central edge of $P_i$ as specified above is
$x_{i+\CL{m/2}}x_{i-\FL{m/2}}$, which we designate as $e_i$.
Note that $\VEC e1m$ is a perfect matching in $K_{2m}$.
The example with $m=4$ decomposes $K_8$ into the spanning paths
$P_1,\ldots, P_4$, where $P_1=\la x_1x_2x_8x_3x_7x_4x_6x_5\ra$,
$P_2=\la x_2x_3x_1x_4x_8x_5x_7x_6\ra$, $P_3=\la x_3x_4x_2x_5x_1x_6x_8x_7\ra$,
and $P_4=\la x_4x_5x_3x_6x_2x_7x_1x_8\ra$.  The matching consisting of the
central edges is $\{x_3x_7, x_4x_8, x_5x_1, x_6x_2\}$.
Note also that every orientation of a path is a bar visibility digraph.

\begin{theorem}\label{n/4}
The bar visibility number of the transitive tournament $T_n$ is at most
$\CL{n/4}$.
\end{theorem}
\begin{proof}
By Lemma~\ref{n+2}, it suffices to prove $b(T_n)\le m$ when $n=4m$.  We aim to
decompose $T_n$ into $m$ bar visibility digraphs, each represented using one
bar per vertex; this yields $b(T_n)\le m$.  Index the vertices of $T_n$ as
$\VEC v1n$ so that the arcs are $\{v_iv_j\st i<j\}$.  Partition the vertex set
into two sets $A$ and $B$, where $A=\{\VEC v1m\}\cup\{\VEC v{3m+1}{4m}\}$ and
$B=\{\VEC v{m+1}{3m}\}$.

The subtournaments $T_n[A]$ and $T_n[B]$ induced by $A$ and $B$ are isomorphic
to $T_{2m}$.  By Lemma~\ref{h}, they decompose into orientations of $m$ paths,
which we call $\VEC P1m$ in $T_n[A]$ and $\VEC Q1m$ in $T_n[B]$.  These paths
inherit orientations from $T_{4m}$.  In order to express them in
the form~\eqref{p}, in $\VEC P1m$ we view $\VEC v1m$ as $\VEC x1m$ and
$\VEC v{3m+1}{4m}$ as $\VEC x{m+1}{2m}$.  In $\VEC Q1m$, we view
$\VEC v{m+1}{3m}$ as $\VEC x1{2m}$ in order.

The remaining arcs form an orientation of the complete bipartite graph
$K_{2m,2m}$ with parts $A$ and $B$.  The arcs are oriented from $\VEC v1m$ in
$A$ to all of $B$ and from all of $B$ to $\VEC v{3m+1}{4m}$ in $A$.

Recall that the central arcs $\VEC e1m$ of the paths $\VEC P1m$ form a perfect
matching on $A$.  Let $G_i$ be the digraph obtained by joining both endpoints
of $e_i$ to all the vertices of $Q_i$, inheriting the orientation from $T_n$.
As illustrated in Figure~\ref{f1}, $G_i$ is a planar digraph: we place the
vertices of $Q_i$ on a horizontal axis between the vertices of $e_i$, with the
rest of $P_i$ extending from the central arc $e_i$.  Figure~\ref{f1} shows the
decomposition $\{G_1,\ldots, G_4\}$ for $T_{16}$.

To show that $G_i$ is a bar visibility digraph, we apply Theorem~\ref{barvdig}.
Note first that each arc $e_i$ has one endpoint in $\VEC v1m$ and one endpoint
in $\VEC v{3m+1}{4m}$.  This means that every vertex in $B$ is neither a source
nor a sink in $G_i$.  In the figure, we add $s$ to the left and $t$ to the
right.  Since sources and sinks in $G_i$ lie along $P_i$, we can add arcs from
$s$ to the sources and from the sinks to $t$, plus the arc $st$, while
maintaining planarity.  Hence by Theorem~\ref{barvdig} $G_i$ is a bar
visibility graph, as desired.
\end{proof}

\begin{figure}[h]
\begin{center}
\begin{tikzpicture}
[p/.style={circle,draw=black,inner sep=1.4pt},>=triangle 45,xscale=0.85,yscale=1.1]

\node(1) at(0.5,-2)[p]{};
\node(2) at(1.5,-2)[p]{};
\node(16) at(2.5,-2)[p]{};
\node(3) at(3.5,-2)[p]{};
\node(15) at(3.5,2)[p]{};
\node(4) at(2.5,2)[p]{};
\node(14) at(1.5,2)[p]{};
\node(13) at(0.5,2)[p]{};

\node(5) at(7,0)[p]{};
\node(6) at(6,0)[p]{};
\node(12) at(5,0)[p]{};
\node(7) at(4,0)[p]{};
\node(11) at(3,0)[p]{};
\node(8) at(2,0)[p]{};
\node(10) at(1,0)[p]{};
\node(9) at(0,0)[p]{};

\draw (0.5,-2.3) node {\small 1}
      (1.5,-2.3) node {\small 2}
      (2.5,-2.3) node {\small 16}
      (3.5,-2.3) node {\small 3}
      (0.5,2.3) node {\small 13}
      (1.5,2.3)   node {\small 14}
      (2.5,2.3) node {\small 4}
      (3.5,2.3) node {\small 15}
      (-0.15,0.25) node {\small 9}
      (0.88,0.25) node {\small 10}
      (1.8,0.25)   node {\small 8}
      (2.8,0.25) node {\small 11}
      (4.2,0.25) node {\small 7}
      (5.2,0.25) node {\small 12}
      (6.1,0.25) node {\small 6}
      (7.1,0.25) node {\small 5};

\draw(1)--(2);\draw(2)--(16);\draw(16)--(3); \draw(15)--(4);\draw(4)--(14);\draw(14)--(13);
\draw(5)--(6); \draw(6)--(12);\draw(12)--(7); \draw(7)--(11);\draw(11)--(8);\draw(8)--(10);\draw(10)--(9);
\node(12x) at ($(1)!.8!(2)$) {};\node(216x) at ($(2)!.8!(16)$) {};\node(316x) at ($(3)!.8!(16)$) {};\node(415x) at ($(4)!.8!(15)$) {};
\node(414x) at ($(4)!.8!(14)$) {};\node(1314x) at ($(13)!.8!(14)$) {};\node(56x) at ($(5)!.8!(6)$) {};\node(612x) at ($(6)!.8!(12)$) {};
\node(712x) at ($(7)!.8!(12)$) {};\node(711x) at ($(7)!.8!(11)$) {};\node(811x) at ($(8)!.8!(11)$) {};\node(810x) at ($(8)!.8!(10)$) {};
\node(910x) at ($(9)!.8!(10)$) {};
\draw[->](1)--(12x); \draw[->](2)--(216x);\draw[->](3)--(316x); \draw[->](4)--(415x);\draw[->](4)--(414x);\draw[->](13)--(1314x);
\draw[->](5)--(56x); \draw[->](6)--(612x);\draw[->](7)--(712x); \draw[->](7)--(711x);\draw[->](8)--(811x);\draw[->](8)--(810x);\draw[->](9)--(910x);

\draw(3)--(5);\node(35x) at ($(3)!.65!(5)$) {};\draw[->](3)--(35x);
\draw(3)--(6);\node(36x) at ($(3)!.69!(6)$) {};\draw[->](3)--(36x);
\draw(3)--(12);\node(312x) at ($(3)!.7!(12)$) {};\draw[->](3)--(312x);
\draw(3)--(7);\node(37x) at ($(3)!.7!(7)$) {};\draw[->](3)--(37x);
\draw(3)--(11);\node(311x) at ($(3)!.7!(11)$) {};\draw[->](3)--(311x);
\draw(3)--(8);\node(38x) at ($(3)!.7!(8)$) {};\draw[->](3)--(38x);
\draw(3)--(10);\node(310x) at ($(3)!.69!(10)$) {};\draw[->](3)--(310x);
\draw(3)--(9);\node(39x) at ($(3)!.65!(9)$) {};\draw[->](3)--(39x);

\draw(15)--(5); \draw(15)--(6);\draw(15)--(12); \draw(15)--(7);\draw(15)--(11);\draw(15)--(8);\draw(15)--(10);\draw(15)--(9);
\node(515x) at ($(5)!.66!(15)$) {};\node(615x) at ($(6)!.69!(15)$) {};\node(1215x) at ($(12)!.7!(15)$) {};\node(715x) at ($(7)!.7!(15)$) {};
\node(1115x) at ($(11)!.7!(15)$) {};\node(815x) at ($(8)!.7!(15)$) {};\node(1015x) at ($(10)!.69!(15)$) {};\node(915x) at ($(9)!.66!(15)$) {};
\draw[<-](515x)--(5); \draw[<-](615x)--(6);\draw[<-](1215x)--(12); \draw[<-](715x)--(7);\draw[<-](1115x)--(11);\draw[<-](815x)--(8);\draw[<-](1015x)--(10);\draw[<-](915x)--(9);

\draw[->](7.63,0)--(7.63,0.1);
\draw (3) .. controls +(-5:5.5cm) and +(5:5.5cm) .. (15);

\begin{scope}[xshift=10cm]
\node(2) at(0.5,-2)[p]{};
\node(3) at(1.5,-2)[p]{};
\node(1) at(2.5,-2)[p]{};
\node(4) at(3.5,-2)[p]{};
\node(16) at(3.5,2)[p]{};
\node(13) at(2.5,2)[p]{};
\node(15) at(1.5,2)[p]{};
\node(14) at(0.5,2)[p]{};

\node(6) at(7,0)[p]{};
\node(7) at(6,0)[p]{};
\node(5) at(5,0)[p]{};
\node(8) at(4,0)[p]{};
\node(12) at(3,0)[p]{};
\node(9) at(2,0)[p]{};
\node(11) at(1,0)[p]{};
\node(10) at(0,0)[p]{};

\draw (0.5,-2.3) node {\small 2}
      (1.5,-2.3) node {\small 3}
      (2.5,-2.3) node {\small 1}
      (3.5,-2.3) node {\small 4}
      (0.5,2.3) node {\small 14}
      (1.5,2.3)   node {\small 15}
      (2.5,2.3) node {\small 13}
      (3.5,2.3) node {\small 16}
      (-0.15,0.25) node {\small 10}
      (0.88,0.25) node {\small 11}
      (1.8,0.25)   node {\small 9}
      (2.75,0.25) node {\small 12}
      (4.2,0.25) node {\small 8}
      (5.1,0.25) node {\small 5}
      (6.1,0.25) node {\small 7}
      (7.1,0.25) node {\small 6};

\draw(2)--(3);\draw(3)--(1);\draw(1)--(4); \draw(16)--(13);\draw(13)--(15);\draw(15)--(14);
\draw(6)--(7); \draw(7)--(5);\draw(5)--(8); \draw(8)--(12);\draw(12)--(9);\draw(9)--(11);\draw(11)--(10);
\node(23x) at ($(2)!.8!(3)$) {};\node(13x) at ($(1)!.8!(3)$) {};\node(14x) at ($(1)!.8!(4)$) {};\node(1613x) at ($(13)!.8!(16)$) {};
\node(1315x) at ($(13)!.8!(15)$) {};\node(1415x) at ($(14)!.8!(15)$) {};\node(67x) at ($(6)!.8!(7)$) {};\node(57x) at ($(5)!.8!(7)$) {};
\node(58x) at ($(5)!.8!(8)$) {};\node(812x) at ($(8)!.8!(12)$) {};\node(912x) at ($(9)!.8!(12)$) {};\node(911x) at ($(9)!.8!(11)$) {};
\node(1011x) at ($(10)!.8!(11)$) {};
\draw[->](2)--(23x); \draw[<-](13x)--(1);\draw[->](1)--(14x); \draw[<-](1613x)--(13);\draw[->](13)--(1315x);\draw[<-](1415x)--(14);
\draw[->](6)--(67x); \draw[<-](57x)--(5);\draw[->](5)--(58x); \draw[->](8)--(812x);\draw[<-](912x)--(9);\draw[->](9)--(911x);\draw[<-](1011x)--(10);

\draw(4)--(5);\node(45x) at ($(4)!.7!(5)$) {};\draw[->](4)--(45x);
\draw(4)--(6);\node(46x) at ($(4)!.65!(6)$) {};\draw[->](4)--(46x);
\draw(4)--(12);\node(412x) at ($(4)!.7!(12)$) {};\draw[->](4)--(412x);
\draw(4)--(7);\node(47x) at ($(4)!.69!(7)$) {};\draw[->](4)--(47x);
\draw(4)--(11);\node(411x) at ($(4)!.69!(11)$) {};\draw[->](4)--(411x);
\draw(4)--(8);\node(48x) at ($(4)!.7!(8)$) {};\draw[->](4)--(48x);
\draw(4)--(10);\node(410x) at ($(4)!.65!(10)$) {};\draw[->](4)--(410x);
\draw(4)--(9);\node(49x) at ($(4)!.7!(9)$) {};\draw[->](4)--(49x);

\draw(16)--(5); \draw(16)--(6);\draw(16)--(12); \draw(16)--(7);\draw(16)--(11);\draw(16)--(8);\draw(16)--(10);\draw(16)--(9);
\node(165x) at ($(5)!.7!(16)$) {};\node(166x) at ($(6)!.66!(16)$) {};\node(1612x) at ($(12)!.7!(16)$) {};\node(167x) at ($(7)!.69!(16)$) {};
\node(1611x) at ($(11)!.69!(16)$) {};\node(168x) at ($(8)!.7!(16)$) {};\node(1610x) at ($(10)!.66!(16)$) {};\node(169x) at ($(9)!.7!(16)$) {};
\draw[<-](165x)--(5); \draw[<-](166x)--(6);\draw[<-](1612x)--(12); \draw[<-](167x)--(7);\draw[<-](1611x)--(11);\draw[<-](168x)--(8);\draw[<-](1610x)--(10);\draw[<-](169x)--(9);

\draw[->](7.63,0)--(7.63,0.1);
\draw (4) .. controls +(-5:5.5cm) and +(5:5.5cm) .. (16);

\end{scope}
\end{tikzpicture}
\\ (a) The subgraph $G_1$ \quad\quad\quad\quad\quad\quad\quad\quad\quad\quad
(b)  The subgraph $G_2$
\end{center}


\begin{center}
\begin{tikzpicture}
[p/.style={circle,draw=black,inner sep=1.4pt},>=triangle 45,xscale=0.85,yscale=1.1]
\node(3) at(0.5,-2)[p]{};
\node(4) at(1.5,-2)[p]{};
\node(2) at(2.5,-2)[p]{};
\node(13) at(3.5,-2)[p]{};
\node(1) at(3.5,2)[p]{};
\node(14) at(2.5,2)[p]{};
\node(16) at(1.5,2)[p]{};
\node(15) at(0.5,2)[p]{};

\node(7) at(7,0)[p]{};
\node(8) at(6,0)[p]{};
\node(6) at(5,0)[p]{};
\node(9) at(4,0)[p]{};
\node(5) at(3,0)[p]{};
\node(10) at(2,0)[p]{};
\node(12) at(1,0)[p]{};
\node(11) at(0,0)[p]{};

\draw (0.5,-2.3) node {\small 3}
      (1.5,-2.3) node {\small 4}
      (2.5,-2.3) node {\small 2}
      (3.5,-2.3) node {\small 13}
      (0.5,2.3) node {\small 15}
      (1.5,2.3)   node {\small 16}
      (2.5,2.3) node {\small 14}
      (3.5,2.3) node {\small 1}
      (-0.15,0.25) node {\small 11}
      (0.88,0.25) node {\small 12}
      (1.8,0.25)   node {\small 10}
      (2.8,0.25) node {\small 5}
      (4.2,0.25) node {\small 9}
      (5.1,0.25) node {\small 6}
      (6.1,0.25) node {\small 8}
      (7.1,0.25) node {\small 7};

\draw(3)--(4);\draw(4)--(2);\draw(2)--(13); \draw(1)--(14);\draw(14)--(16);\draw(16)--(15);
\draw(7)--(8); \draw(8)--(6);\draw(6)--(9); \draw(9)--(5);\draw(5)--(10);\draw(10)--(12);\draw(12)--(11);
\node(34x) at ($(3)!.8!(4)$) {};\node(24x) at ($(2)!.8!(4)$) {};\node(213x) at ($(2)!.8!(13)$) {};\node(114x) at ($(1)!.8!(14)$) {};
\node(1416x) at ($(14)!.8!(16)$) {};\node(1516x) at ($(15)!.8!(16)$) {};\node(78x) at ($(7)!.8!(8)$) {};\node(68x) at ($(6)!.8!(8)$) {};
\node(69x) at ($(6)!.8!(9)$) {};\node(59x) at ($(5)!.8!(9)$) {};\node(510x) at ($(5)!.8!(10)$) {};\node(1012x) at ($(10)!.8!(12)$) {};
\node(1112x) at ($(11)!.8!(12)$) {};
\draw[->](3)--(34x); \draw[->](2)--(24x);\draw[->](2)--(213x); \draw[->](1)--(114x);\draw[->](14)--(1416x);\draw[->](15)--(1516x);
\draw[->](7)--(78x); \draw[->](6)--(68x);\draw[->](11)--(1112x); \draw[->](6)--(69x);\draw[->](5)--(59x);\draw[->](5)--(510x);\draw[->](10)--(1012x);

\draw(13)--(7);\node(713x) at ($(7)!.59!(13)$) {};\draw[->](7)--(713x);
\draw(13)--(8);\node(813x) at ($(8)!.6!(13)$) {};\draw[->](8)--(813x);
\draw(13)--(6);\node(613x) at ($(6)!.6!(13)$) {};\draw[->](6)--(613x);
\draw(13)--(9);\node(913x) at ($(9)!.6!(13)$) {};\draw[->](9)--(913x);
\draw(13)--(5);\node(513x) at ($(5)!.6!(13)$) {};\draw[->](5)--(513x);
\draw(13)--(10);\node(1013x) at ($(10)!.6!(13)$) {};\draw[->](10)--(1013x);
\draw(13)--(12);\node(1213x) at ($(12)!.6!(13)$) {};\draw[->](12)--(1213x);
\draw(13)--(11);\node(1113x) at ($(11)!.59!(13)$) {};\draw[->](11)--(1113x);

\draw(1)--(7); \draw(1)--(8);\draw(1)--(6); \draw(1)--(9);\draw(1)--(5);\draw(1)--(10);\draw(1)--(12);\draw(1)--(11);
\node(17x) at ($(1)!.56!(7)$) {};\node(18x) at ($(1)!.59!(8)$) {};\node(16x) at ($(1)!.6!(6)$) {};\node(19x) at ($(1)!.6!(9)$) {};
\node(15x) at ($(1)!.6!(5)$) {};\node(110x) at ($(1)!.6!(10)$) {};\node(112x) at ($(1)!.59!(12)$) {};\node(111x) at ($(1)!.56!(11)$) {};
\draw[<-](17x)--(1); \draw[<-](18x)--(1);\draw[<-](16x)--(1); \draw[<-](19x)--(1);\draw[<-](15x)--(1);\draw[<-](110x)--(1);\draw[<-](112x)--(1);\draw[<-](111x)--(1);

\draw[<-](7.63,0)--(7.63,0.1);
\draw (13) .. controls +(-5:5.5cm) and +(5:5.5cm) .. (1);

\begin{scope}[xshift=10cm]
\node(4) at(0.5,-2)[p]{};
\node(13) at(1.5,-2)[p]{};
\node(3) at(2.5,-2)[p]{};
\node(14) at(3.5,-2)[p]{};
\node(2) at(3.5,2)[p]{};
\node(15) at(2.5,2)[p]{};
\node(1) at(1.5,2)[p]{};
\node(16) at(0.5,2)[p]{};

\node(8) at(7,0)[p]{};
\node(9) at(6,0)[p]{};
\node(7) at(5,0)[p]{};
\node(10) at(4,0)[p]{};
\node(6) at(3,0)[p]{};
\node(11) at(2,0)[p]{};
\node(5) at(1,0)[p]{};
\node(12) at(0,0)[p]{};

\draw (0.5,-2.3) node {\small 4}
      (1.5,-2.3) node {\small 13}
      (2.5,-2.3) node {\small 13}
      (3.5,-2.3) node {\small 14}
      (0.5,2.3) node {\small 16}
      (1.5,2.3)   node {\small 1}
      (2.5,2.3) node {\small 15}
      (3.5,2.3) node {\small 2}
      (-0.15,0.25) node {\small 12}
      (0.88,0.25) node {\small 5}
      (1.8,0.25)   node {\small 11}
      (2.8,0.25) node {\small 6}
      (4.2,0.25) node {\small 10}
      (5.1,0.25) node {\small 7}
      (6.1,0.25) node {\small 9}
      (7.1,0.25) node {\small 8};

\draw(4)--(13); \draw(13)--(3);\draw(3)--(14); \draw(2)--(15);\draw(15)--(1);\draw(1)--(16);
\draw(8)--(9); \draw(9)--(7);\draw(7)--(10); \draw(10)--(6);\draw(6)--(11);\draw(11)--(5);\draw(5)--(12);
\node(413x) at ($(4)!.8!(13)$) {};\node(313x) at ($(3)!.8!(13)$) {};\node(314x) at ($(3)!.8!(14)$) {};\node(215x) at ($(2)!.8!(15)$) {};
\node(115x) at ($(1)!.8!(15)$) {};\node(116x) at ($(1)!.8!(16)$) {};\node(89x) at ($(8)!.8!(9)$) {};\node(79x) at ($(7)!.8!(9)$) {};
\node(710x) at ($(7)!.8!(10)$) {};\node(610x) at ($(6)!.8!(10)$) {};\node(611x) at ($(6)!.8!(11)$) {};\node(511x) at ($(5)!.8!(11)$) {};
\node(512x) at ($(5)!.8!(12)$) {};
\draw[->](4)--(413x); \draw[->](3)--(313x);\draw[->](3)--(314x); \draw[->](2)--(215x);\draw[->](1)--(115x);\draw[->](1)--(116x);
\draw[->](8)--(89x); \draw[->](7)--(79x);\draw[->](7)--(710x); \draw[->](6)--(610x);\draw[->](6)--(611x);\draw[->](5)--(511x);\draw[->](5)--(512x);

\draw(2)--(5); \draw(2)--(6);\draw(2)--(12); \draw(2)--(7);\draw(2)--(11);\draw(2)--(8);\draw(2)--(10);\draw(2)--(9);\draw(14)--(5); \draw(14)--(6);\draw(14)--(12); \draw(14)--(7);\draw(14)--(11);\draw(14)--(8);\draw(14)--(10);\draw(14)--(9);

\node(25x) at ($(2)!.59!(5)$) {};\draw[->](2)--(25x);
\node(26x) at ($(2)!.6!(6)$) {};\draw[->](2)--(26x);
\node(212x) at ($(2)!.56!(12)$) {};\draw[->](2)--(212x);
\node(27x) at ($(2)!.6!(7)$) {};\draw[->](2)--(27x);
\node(211x) at ($(2)!.6!(11)$) {};\draw[->](2)--(211x);
\node(28x) at ($(2)!.56!(8)$) {};\draw[->](2)--(28x);
\node(210x) at ($(2)!.6!(10)$) {};\draw[->](2)--(210x);
\node(29x) at ($(2)!.59!(9)$) {};\draw[->](2)--(29x);

\node(514x) at ($(5)!.6!(14)$) {};\node(614x) at ($(6)!.6!(14)$) {};\node(1214x) at ($(12)!.59!(14)$) {};\node(714x) at ($(7)!.6!(14)$) {};
\node(1114x) at ($(11)!.6!(14)$) {};\node(814x) at ($(8)!.59!(14)$) {};\node(1014x) at ($(10)!.6!(14)$) {};\node(914x) at ($(9)!.6!(14)$) {};
\draw[<-](514x)--(5); \draw[<-](614x)--(6);\draw[<-](1214x)--(12); \draw[<-](714x)--(7);\draw[<-](1114x)--(11);\draw[<-](814x)--(8);\draw[<-](1014x)--(10);\draw[<-](914x)--(9);

\draw[<-](7.63,0)--(7.63,0.1);
\draw (14) .. controls +(-5:5.5cm) and +(5:5.5cm) .. (2);
\end{scope}
\end{tikzpicture}
\\ (c) The subgraph $G_3$ \quad\quad\quad\quad\quad\quad\quad\quad\quad\quad
(d)  The subgraph $G_4$
\end{center}
\caption{A decomposition of $T_{16}$ in which each subgraph is a bar visibility digraph.}
\label{f1}
\end{figure}
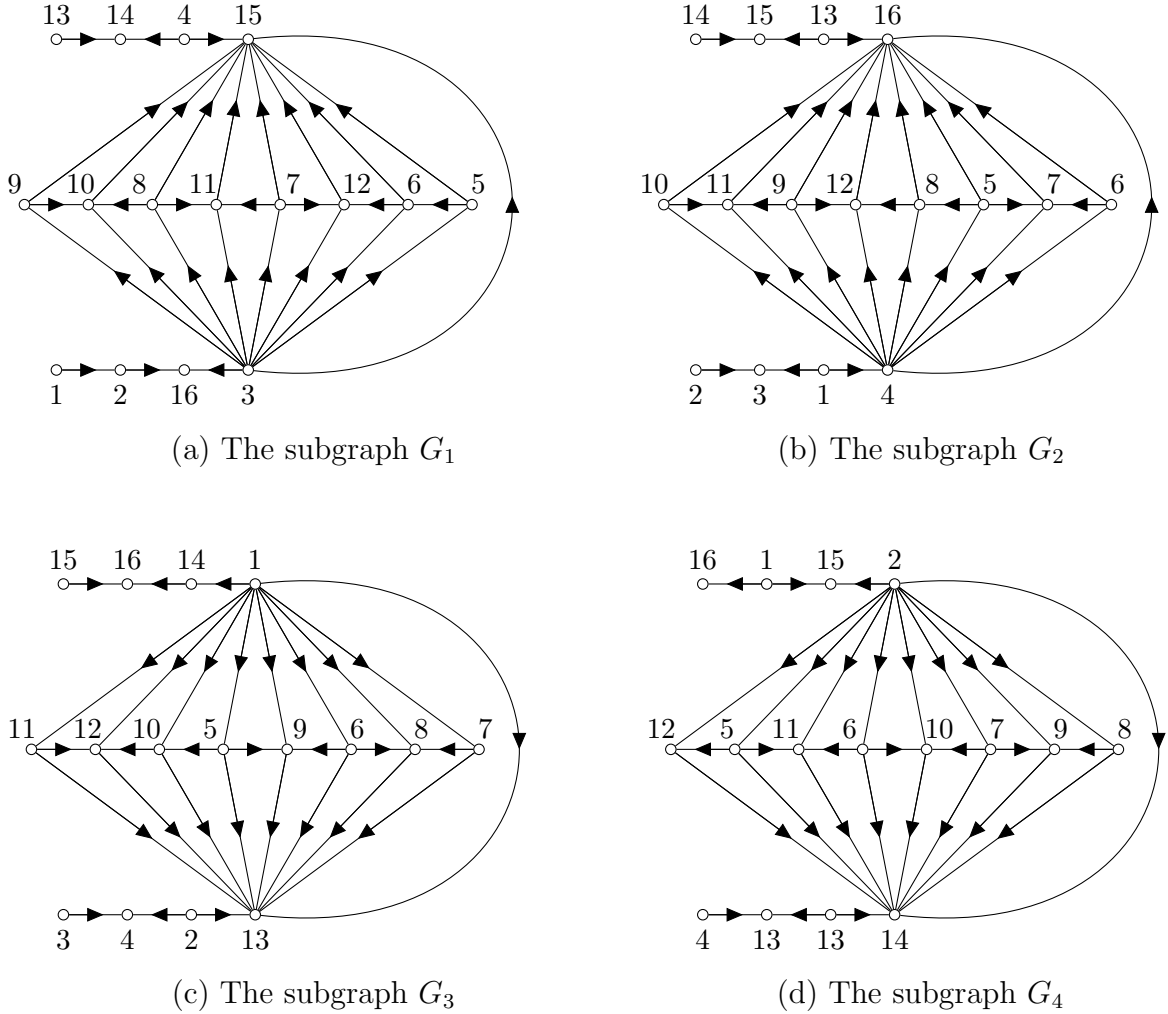

Theorem~\ref{n/4} yields $b(T_{11})\le b(T_{12})\le 3$, which also follows
from the construction in~\cite{ABHW2013} for $b(T_{15})\le 3$.  Proving
Conjecture~\ref{con1} that $b(T_{11})=3$ requires the lower bound, which will
follow from our results in the next section.  They also yield $b(T_{17})\ge4$,
which with Lemma~\ref{n+2} and $b(T_{15})\le 3$ from~\cite{ABHW2013} implies
$b(T_{17})=4$.  It remains open whether $b(T_{16})$ is $3$ or $4$.

\section{$b(T_n)/n$: Convergence and a Lower Bound}
In this section we prove that $b(T_n)/n$ converges as $n\to\infty$ and derive a
nontrivial lower bound on it that implies $b(T_{11})\ge3$ and $b(T_{17})\ge4$.

\begin{observation}[\cite{ABHW2013}]
If $G$ is a digraph with underlying graph $\hat{G}$, then $b(G)\geq b(\hat{G})$.
\end{observation}
\begin{proof}
A $t$-bar representation of $G$ is also a $t$-bar
representation of its underlying graph.
\end{proof}

Chang et al. \cite{Chang04} proved that the complete graph $K_n$ has bar
visibility number $\CL{n/6}$ for $n\geq 7$; thus also $b(T_n)\ge\CL{n/6}$ for
$n\ge7$.  With Theorem~\ref{upper}, it follows that
$1/6\le b(T_n)/n\le 3/14+O(1/n)$.  To prove that $b(T_n)/n$ converges,
we need the following lemma, which also yields the upper bound of $3/14$
from $b(T_{15})=3$.

\begin{lemma}[\cite{ABHW2013}]\label{upper2}
If $b(T_l)=t$ for some $l$, then $b(T_n)\leq \frac{tn}{l-1}+O(1)$
for sufficiently large $n$.
\end{lemma}

This lemma is based on the famous result of Wilson~\cite{Wilson} implying
that when $n$ is sufficiently large, there exists $m$ with $n\le m\le n+c_l$
such that $K_m$ decomposes into copies of $K_l$ (called a {\it Steiner system}).
In $T_n$, the vertex sets of these copies induce copies of $T_l$, which has
a bar visibility representation using at most $t$ bars per vertex.  In the
decomposition, each copy containing a vertex $v$ uses $l-1$ of the edges
incident to it in $K_m$, so each vertex appears in $(m-1)/(l-1)$ copies in the
decomposition.  We thus obtain a representation of $T_m$ using at most
$t(m-1)/(l-1)$ bars per vertex, and then deleting bars for any $m-n$ vertices
does not introduce unwanted visibilities.  Thus $b(T_n)\le tn/(l-1)+O(1)$.

\begin{theorem}\label{limit}
$b(T_n)/n$ converges.
\end{theorem}

\begin{proof}
Let $a=\liminf b(T_n)/n$ and $b=\limsup b(T_n)/n$.
If $b(T_n)/n$ does not converge, then $a<b$.
By the definitions of $\liminf$ and $\limsup$, there is a positive integer $l$
with $l>\frac{3(a+b)}{b-a}+1$ such that $b(T_l)/l=c$, where $a\le c<(a+b)/2$.
That is, $b(T_l)=cl$.  By Lemma~\ref{upper2}, $b(T_n)\leq \frac{cln}{l-1}+O(1)$
for sufficiently large $n$.  For sufficiently large $n$, we then have
\begin{eqnarray*}
\ \frac{b(T_n)}{n} & \leq & \frac{c\cdot l}{l-1}+o(1)
~ < ~ \frac{a+b}{2}+\frac{a+b}{2}\cdot \frac{1}{l-1} \\
& < & \frac{a+b}{2}+\frac{a+b}{2}\cdot \frac{1}{\frac{3(a+b)}{b-a}+1-1}
~ = ~ \frac{a+2b}{3}
~ < ~ b,
\end{eqnarray*}
which contradicts $\limsup b(T_n)/n=b$.
\end{proof}

The lower bound uses an undirected graph associated with a $t$-bar visibility
representation.

\begin{definition}\label{derived}
The {\it derived graph} of a $t$-bar visibility representation is a plane
graph obtained by introducing an edge for each pair of bars that see each other
along an unblocked channel (omitting loops) and then shrinks each bar to a
point, keeping its edges.
\end{definition}

Given a $k$-bar visibility representation of $K_n$, the derived graph is
a planar graph $H$ with at most $kn$ vertices and at least $\CH n2$ edges.
Euler's Formula then requires $\CH n2\le 3kn-6$, which simplifies to
$k>(n-1)/6$ since $k$ is an integer.  We improve on this lower bound by showing
that at least $k^2-k$ of the edges in $H$ duplicate visibilities and hence are
wasted.

\begin{theorem}\label{lower}
The transitive tournament $T_n$ on $n$ vertices satisfies
$$b(T_n)\geq \frac{3n-5-\sqrt{7n^2-28n+25}}{2}>\frac{3-\sqrt{7}}{2}n+\sqrt{7}-\frac{5}{2}.$$
Therefore $\lim b(T_n)/n\ge (3-\sqrt{7})/{2}\approx 0.177124$.
\end{theorem}


\begin{proof}
Let $k=b(T_n)$; by Lemma~\ref{n+2}, $k\le n/2$.
Begin with a $k$-bar visibility representation of $T_n$ giving $k$ bars to each
vertex.  Index the vertices as $\VEC v1n$ so that all arcs
are oriented from $v_i$ to $v_j$ with $i<j$.  With this vertex ordering, we can
shift bars vertically so that each bar has vertical coordinate equal to the
index of its assigned vertex.  We can also combine the bars for $v_1$ into a
single bar and those for $v_n$ into a single bar and extend each to have the
leftmost left endpoint and rightmost right endpoint among all bars.

The result is again a $k$-bar visibility representation of $T_n$, using
altogether $k(n-2)+2$ bars.  Its derived graph $H$ is a planar graph with
$k(n-2)+2$ vertices.  Thus $H$ has at most $3k(n-2)$ edges.
To derive a lower bound on $\C{E(H)}$, we begin by selecting $\CH n2$ edges
consisting of one each from level $i$ to level $j$ for all $i$ and $j$
such that $1\le i<j\le n$.  Next we find extra edges.

For each $j$ with $2\le j\le k$, there are $k$ bars at level $j$.
Because we have extended the bar at level $1$ to be leftmost and rightmost,
each bar at level $j$ is seen by at least one bar from below.  Hence the
vertices representing $v_j$ in $H$ together have at least $k$ edges entering
them in $H$.  However, in $T_n$ there are only $j-1$ edges entering $v_j$,
and we have already counted that many.  Hence there are at least $k-j+1$
extra edges entering the vertices at level $j$.  Summing over all such $j$,
there are at least $k(k-1)/2$ extra edges entering these vertices.

A symmetric argument applies to edges leaving vertices near the top.  For each
$i$ with $n-k+1\le i\le n-1$, there are at least $k-(n-i)$ extra edges leaving
the vertices at level $i$.  Summing over all such $i$, we have found at least
$k(k-1)/2$ extra edges from vertices at level $i$ to higher vertices.
Since $k\le n/2$, these two sets of extra edges are disjoint, and $H$ has at
least $k(k-1)$ edges beyond those selected to represent actual edges of $T_n$.

We now have the inequality
$$\frac{n(n-1)}{2}+k^2-k\leq 3k(n-2),$$
which simplifies to $2k^2-(6n-10)k+n^2-n\leq 0$.  Solving the quadratic
inequality yields
$$k\ge\frac{3n-5-\sqrt{7n^2-28n+25}}{2}>\frac{3-\sqrt7}{2}n+\sqrt7-\frac 52.$$
This completes the proof.
\end{proof}

\begin{corollary}
$b(T_{11})=b(T_{12})=3$ and $b(T_{17})=4$.
\end{corollary}
\begin{proof}
Set $n=11$ and $n=17$ in the formula of Theorem~\ref{lower} to obtain the lower
bounds.  We have
$b(T_{11})\geq \frac{3\cdot11-5-\sqrt{7\cdot 11^2-28\cdot 11+25}}{2}=14-\sqrt{141}\approx 2.1257$
and
$b(T_{17})\geq \frac{3\cdot17-5-\sqrt{7\cdot 17^2-28\cdot
17+25}}{2}=23-\sqrt{393}\approx 3.1758$,
so $b(T_{11})\ge3$ and $b(T_{17})\ge4$.
For the upper bounds, we cite Lemma~\ref{n+2} and the known values
$b(T_{13})=b(T_{15})=3$ from Theorem~\ref{smalln}.
\end{proof}

\section{Bar visibility numbers of a graph and its orientations}

In this section, we provide an undirected graph $\hG$ having a orientation $G$
with $b(G)>2b(\hG)$, thereby disproving Conjecture~\ref{con2}.  Nevertheless,
upper bounds in terms of $b(\hG)$ do hold.

The tool we use for the construction is the interval number of a graph.
A {\it $t$-interval representation} of a graph $H$ assigns to each vertex
$v\in V(H)$ at most $t$ intervals in $\mathbb{R}$ so that $uv\in E(H)$ if and
only if some interval assigned to $u$ intersects some interval assigned to $v$.
The {\it interval number $i(H)$} of $H$ is the minimum $t$ such that $H$ has a
$t$-interval representation.  A $t$-interval representation of $H$ has
{\it depth $2$} if no point on the real line lies in intervals assigned to more
than two vertices.  The proof of the following theorem published originally
in~\cite{SW1983} was flawed, but a different and shorter proof has now been
published by Gu\'egan, Knauer, Rollin, and Ueckerdt~\cite{GKRU}.

\begin{theorem}[\cite{GKRU}]\label{interval}
Every planar graph has interval number at most $3$, and this is sharp.
\end{theorem}

In \cite{SW1983}, Scheinerman and West showed that the planar graph $\hG_1$ in
Figure~\ref{Graph G1} has interval number $3$.  This graph arises by adding
a pendent edge at each vertex in the larger part of the complete bipartite
graph $K_{2,9}$.

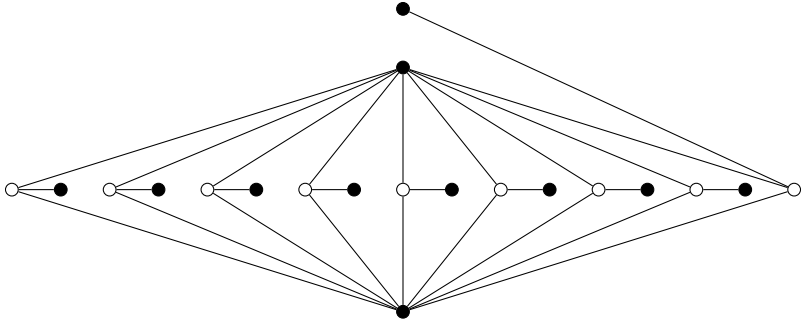
\begin{figure}[htbp]
\centering
\begin{tikzpicture} [scale = 0.65,
	a/.style={circle, draw,  fill=black, inner sep=1.7pt},
	b/.style={circle, draw,  inner sep=1.7pt}]

	\node[a] (a10) at (0,-2.5)  {};
	\node[a] (a11) at (0,2.5)   {};
	\node[a] (a1) at (-7,0)  {};
	\node[a] (a2) at (-5,0)  {};
	\node[a] (a3) at (-3,0)   {};
	\node[a] (a4) at (-1,0)  {};
	\node[a] (a5) at (1,0)  {};
	\node[a] (a6) at (3,0)   {};
	\node[a] (a7) at (5,0)  {};
	\node[a] (a8) at (7,0)  {};
	\node[a] (a9) at (0,3.7)   {};

	\node[b] (b1) at (-8,0)    {};
	\node[b] (b2) at (-6,0)    {};
	\node[b] (b3) at (-4,0)    {};
	\node[b] (b4) at (-2,0)   {};
	\node[b] (b5) at (0,0)    {};
	\node[b] (b6) at (2,0)   {};
	\node[b] (b7) at (4,0)   {};
	\node[b] (b8) at (6,0)   {};
	\node[b] (b9) at (8,0)   {};
	
	\draw (a10) to (b1);
	\draw (a10) to (b2);
	\draw (a10) to (b3);
	\draw (a10) to (b4);
	\draw (a10) to (b5);
	\draw (a10) to (b6);
	\draw (a10) to (b7);
	\draw (a10) to (b8);
	\draw (a10) to (b9);
	
	\draw (a11) to (b1);
	\draw (a11) to (b2);
	\draw (a11) to (b3);
	\draw (a11) to (b4);
	\draw (a11) to (b5);
	\draw (a11) to (b6);
	\draw (a11) to (b7);
	\draw (a11) to (b8);
	\draw (a11) to (b9);
	
	\draw (a1) to (b1);
	\draw (a2) to (b2);
	\draw (a3) to (b3);
	\draw (a4) to (b4);
	\draw (a5) to (b5);
	\draw (a6) to (b6);
	\draw (a7) to (b7);
	\draw (a8) to (b8);
	\draw (a9) to (b9);
\end{tikzpicture}
\caption{The graph $\hat{G}_1$ with interval number three.}
\label{Graph G1}
\end{figure}

\begin{lemma}\label{pro1}
If $H'$ is a spanning subgraph of a triangle-free graph $H$,
then  $i(H') \leq i(H)$.
\end{lemma}

\begin{proof}
Let $uv$ be an edge in $H$.  From an $i(H)$-interval representation of $H$,
we obtain an $i(H)$-interval representation of $H-uv$.  Consider intervals
assigned to $u$ and $v$ that intersect.  If one is contained in the other,
delete the smaller interval from the representation.  If they overlap,
shorten each by deleting their intersection.  Since $H$ is triangle-free,
the points in the intersection do not lie in intervals assigned to any other
vertices, so the operation does not delete any other edges from the graph
represented.  The operation also does not add any edges, and the represented
graph remains triangle-free.  Thus iteratively deleting the edges of
$E(H)-E(H')$ in this way yields an $i(H)$-interval representation of $H'$.
\end{proof}

An $t$-interval representation of a bipartite graph $\hG$ with intervals
on a horizontal line can be processed from left to right, shifting intervals up
or down as needed in becoming bars, to produce a $t$-bar visibility
representation of any orientation $G$ of $\hG$.  When $G$ orients all edges of
$\hG$ from one part to the other, no bar can be placed between bars for two
vertices of the other part, and hence the process can be reversed.  This yields
the following statement.

\begin{lemma}[\cite{ABHW2013}]\label{pro2}
If $G$ is an orientation of a bipartite graph $\hG$, then $b(G)\le i(\hG)$,
with equality when all edges are oriented from one part to the other.
\end{lemma}

\begin{theorem}\label{example}
There are digraphs with bar visibility number $3$ whose underlying
undirected graph is a bar visibility graph.
\end{theorem}

\begin{proof}
Let $\hG$ be the graph shown in Figure~\ref{Graph G2}.
This graph is obtained from a cycle with $18$ vertices by adding two vertices
whose neighborhoods are the even-indexed vertices of the cycle.  Since $\hG$ is
a 2-connected planar graph, $b(\hG) = 1$, by Theorem~\ref{barvG}.

Since $\hG$ is triangle-free and the graph $\hG_1$ of Figure~\ref{Graph G1} is
a spanning subgraph of $\hG$, Lemma~\ref{pro1} implies $i(\hG)\ge i(\hG_1)=3$.
On the other hand, Theorem~\ref{interval} yields $i(\hat{G}) \leq 3$. Hence
$i(\hG)=3$.  Let $G$ be an orientation of $\hG$ that orients all edges from one
part to the other.  By Lemma~\ref{pro2}, $b(G) = i(\hG)=3$.

The same argument applies to such orientations of any $2$-connected planar
bipartite graph; here $G$ and $\hG$ provide just one example.
\end{proof}

\begin{figure}[htbp]
\centering
\begin{tikzpicture} [scale = 0.65,
	a/.style={circle, draw, fill=black, inner sep=1.7pt},
	b/.style={circle, draw, inner sep=1.7pt}]

   \node[a] (a10) at (0,-2.5)  {};
	\node[a] (a11) at (0,2.5)   {};
	\node[a] (a1) at (-7,0)  {};
	\node[a] (a2) at (-5,0)  {};
	\node[a] (a3) at (-3,0)   {};
	\node[a] (a4) at (-1,0)  {};
	\node[a] (a5) at (1,0)  {};
	\node[a] (a6) at (3,0)   {};
	\node[a] (a7) at (5,0)  {};
	\node[a] (a8) at (7,0)  {};
	\node[a] (a9) at (0,3.7)   {};

	\node[b] (b1) at (-8,0)    {};
	\node[b] (b2) at (-6,0)    {};
	\node[b] (b3) at (-4,0)    {};
	\node[b] (b4) at (-2,0)   {};
	\node[b] (b5) at (0,0)    {};
	\node[b] (b6) at (2,0)   {};
	\node[b] (b7) at (4,0)   {};
	\node[b] (b8) at (6,0)   {};
	\node[b] (b9) at (8,0)   {};
	
	\draw (a10) to (b1);
	\draw (a10) to (b2);
	\draw (a10) to (b3);
	\draw (a10) to (b4);
	\draw (a10) to (b5);
	\draw (a10) to (b6);
	\draw (a10) to (b7);
	\draw (a10) to (b8);
	\draw (a10) to (b9);
	
	\draw (a11) to (b1);
	\draw (a11) to (b2);
	\draw (a11) to (b3);
	\draw (a11) to (b4);
	\draw (a11) to (b5);
	\draw (a11) to (b6);
	\draw (a11) to (b7);
	\draw (a11) to (b8);
	\draw (a11) to (b9);
	
	\draw (a1) to (b1);  \draw (a1) to (b2);
	\draw (a2) to (b2);  \draw (a2) to (b3);
	\draw (a3) to (b3);  \draw (a3) to (b4);
	\draw (a4) to (b4);  \draw (a4) to (b5);
	\draw (a5) to (b5);  \draw (a5) to (b6);
	\draw (a6) to (b6);  \draw (a6) to (b7);
	\draw (a7) to (b7);  \draw (a7) to (b8);
	\draw (a8) to (b8);  \draw (a8) to (b9);
	\draw (a9) to (b9);  \draw (a9) to (b1);
\end{tikzpicture}
\caption{The graph $\hG$ for Theorem~\ref{example}.}
\label{Graph G2}
\end{figure}
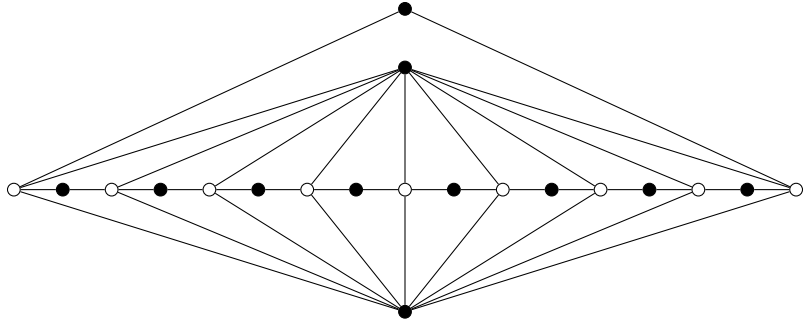

Theorem~\ref{example} disproves Conjecture~\ref{con2}.  Nevertheless, results
from \cite{ABHW2013} yield upper bounds on the bar visibility number of
digraphs from the visibility number of the underlying graph.

\begin{lemma}[\cite{ABHW2013}]\label{lem}
If $G$ is a triangle-free planar digraph, then $b(G)\le3$, and this is sharp.
\end{lemma}

\begin{theorem}[\cite{ABHW2013}]\label{planar-theo}
Let $G$ be an orientation of a planar graph $\hat{G}$.

(i) $b(G)\leq 4$.

(ii) If $\hat{G}$ is triangle-free or contains no subdivision of $K_{2,3}$,
then $b(G)\leq 3$.

(iii) If $\hat{G}$ has girth at least $6$, then $b(G)\leq 2$.
\end{theorem}

\begin{theorem}\label{theo1}
If $G$ is an orientation of a triangle-free graph $\hG$, then $b(G)\le3b(\hG)$,
and this is sharp when $b(\hG)=1$.
\end{theorem}

\begin{proof}
Let $t=b(\hat{G})$.  Because $\hat{G}$ is triangle-free, the derived graph
$H$ of a $t$-bar visibility representation of $\hat{G}$ is also triangle-free.
Orient the edges of $H$ according to the orientation in $G$, obtaining
$\widetilde{H}$.  By Lemma \ref{lem}, $\widetilde{H}$ has a $3$-bar visibility
representation, which produces a $3t$-bar visibility representation of $G$.
Hence $b(G)\le 3t$.

Sharpness is achieved by be the graph $\hG$ in Figure~\ref{Graph G2}, with
$G$ orienting all edges from one part to the other.  The proof of
Theorem~\ref{example} shows $b(G)=3b(\hat{G})$.
\end{proof}

\begin{theorem}\label{theo2}
If $G$ is an orientation of a graph $\hat{G}$, then
$b(G)\leq 4b(\hat{G})$.
%
%
\end{theorem}

\begin{proof}
As in the proof of Theorem \ref{theo1}, the claim follows from
Theorem~\ref{planar-theo}.
\end{proof}

Theorems~\ref{theo1} and~\ref{theo2} rely on the properties in
Lemma~\ref{lem} of being planar and triangle-free.  We do not obtain
similar conclusions for graphs lacking subdivisions of $K_{2,3}$ or having
girth at least $6$, because the planar graph $H$ obtain from the $t$-bar
visibility representation of such a digraph may have a subdivision of
$K_{2,3}$ or a $4$-cycle, respectively.

\section{NP-Completeness}

In this section, we show that recognition of digraphs with bar visibility
number $2$ is NP-complete, by reduction from the Hamiltonian cycle problem in
$3$-regular triangle-free graphs.

The digraphs with bar visibility number $1$ are the bar visibility digraphs,
characterized in Theorem~\ref{barvdig} by whether a certain auxiliary digraph
is planar and has no consistent cycle.  There are linear time algorithms for
testing planarity (see Section 2.7 in \cite{Mohar}).  Also, a digraph has no
consistent cycle if and only if it admits a topological sort, and there exist
polynomial-time algorithms for finding a topological sort (see~\cite{Kahn} for
example).  Hence there is a polynomial-time algorithm for recognition of
digraphs with bar visibility number $1$.

In order to study the recognition problem for bar visibility number $2$,
we define additional aspects of interval representations of graphs, which were
introduced in the previous section.
A $t$-interval representation of a graph $H$ is a {\it displayed representation}
if for each vertex, some assigned interval contains an open interval not
intersecting any other interval.  If the union of a set of intervals in a
$t$-interval representation is a single interval, then we say that these
intervals {\it appear contiguously}.  A graph $H$ is {\it $t$-interval tight}
if $i(H)=t$ and every $t$-interval representation of $H$ assigns $t$ disjoint
intervals to each vertex.

\begin{lemma}[\cite{West84}]\label{Kinterval}
The graph $K_{t^2+t-1,t+1}$ is $t$-interval tight.  If $K_{t^2+t-1,t+1}$ is an
induced subgraph of a graph $H$, then in every $t$-interval representation of
$H$ the intervals for vertices of $K_{t^2+t-1,t+1}$ appear contiguously.
If $u$ and $v$ are any specified vertices from opposite parts of $K_{m,n}$,
then $K_{m,n}$ has a displayed $i(K_{m,n})$-interval representation
in which $u$ and $v$ are assigned the leftmost and rightmost intervals in the
representation, respectively.
\end{lemma}

We introduce concepts for $t$-bar visibility representations of digraphs
analogous to those for $t$-interval representations of graphs.
A $t$-bar visibility representation of a digraph $G$ has {\it depth-$2$} if
every channel intersects at most two bars in the representation.  The
representation is {\it displayed} if for each $v\in V(G)$ there is an unbounded
channel that intersects some bar for $v$ and no other bar.  If the derived
graph of the representation is connected, then we say that the bars in the
representation {\it appear contiguously}.  A digraph $G$ is {\it $t$-bar tight}
if $b(G)=t$ and in every $t$-bar visibility representation each vertex is
assigned $t$ bars.

\begin{lemma}[\cite{ABHW2013}]\label{depth2}
A digraph $G$ has a depth-$2$ $t$-bar visibility representation if and only
if its underlying graph $\hat{G}$ has a depth-$2$ $t$-interval representation.
\end{lemma}

Let $\tK_{m,n}$ denote an orientation of $K_{m,n}$ that orients all its
edges from one part to the other.  Both such orientations have the same
bar visibility number.

\begin{lemma}\label{Kvisibility}
The digraph $\tK_{t^2+t-1,t+1}$ is $t$-bar tight. If
$\tK_{t^2+t-1,t+1}$ is an induced subgraph of a digraph $G$, then in any
$t$-bar visibility representation of $G$ the bars for vertices of
$K_{t^2+t-1,t+1}$ appear contiguously. Furthermore, if $u$ and $v$ are any
specified vertices from opposite parts of $\tK_{m,n}$, then
$\tK_{m,n}$ has a displayed $b(K_{m,n})$-bar visibility representation in
which $u$ and $v$ are assigned the leftmost and rightmost bars in the
representation, respectively.
\end{lemma}

\begin{proof}
Since each vertex in $\tK_{m,n}$ is a source or a sink, all
$b(\tK_{m,n})$-bar visibility representations have depth $2$.
Since $K_{m,n}$ is triangle-free, every $i(K_{m,n})$-interval representation
has depth $2$.  By Lemma~\ref{depth2}, $b(\tK_{m,n})=i(K_{m,n})$.
By Lemma~\ref{Kinterval}, $\tK_{t^2+t-1,t+1}$ is $t$-bar tight.
The claim about specified vertices follows by symmetry.
\end{proof}

Our reduction involves transforming a $3$-regular triangle-free graph $H$ into
a digraph $G$ such that $H$ has a Hamiltonian cycle if and only if $b(G)=2$.

\begin{definition}
Given a $3$-regular triangle-free graph $H$, define a {\it test digraph}
$f(H)$ as follows.  Begin with an arbitrary orientation of $H$.
Add three copies of $\tK_{5,3}$, denoted $H_1$, $H_2$, and $H_3$.
Choose sinks $s_1\in V(H_1)$ and $s_2\in V(H_2)$, choose sources
$t_2\in V(H_2)$ and $t_3\in V(H_3)$, and add the arcs $s_1t_2$ and $s_2t_3$.
For each vertex $v\in V(H)$, add a copy $M_v$ of $\tK_{5,3}$ and an arc
$vv'$ for one vertex $v'\in V(M_v)$.  Also add the arcs $s_1v$ and $vt_2$.
Finally, for one special vertex $z\in V(H)$, add an arc from $z$ to each
vertex in $H_2$ and $H_3$ (we already have $zt_2$; no need for an extra copy).
Also add $s_1x$ and $xt_2$ for each $x\in V(M_z)$.
\end{definition}


\begin{figure}[h]
\begin{center}
\begin{tikzpicture}
[p/.style={circle,draw=black,inner sep=1.4pt},>=triangle 45,q/.style={circle,draw=black,fill=black,inner sep=0.5pt}]
\draw (0,2) rectangle (3,4);\draw (5,2) rectangle (8,4);\draw (10,2) rectangle (13,4);
\node  at (1.5,4.3){$H_1$};\node  at (6.5,4.3){$H_2$};\node  at (11.5,4.3){$H_3$};
\node (s1) at (2,3) [p] {};\node  at (2,3.3){$s_1$};
\node (t2) at (6,3) [p] {};\node  at (6,3.3){$t_2$};
\node (s2) at (7,3) [p] {};\node  at (7,3.3){$s_2$};
\node (t3) at (11,3) [p] {};\node  at (11,3.3){$t_3$};
\draw (s1)--(t2);\draw (s2)--(t3);\node(s1t2) at ($(s1)!.6!(t2)$) {};\draw[->](s1)--(s1t2);
\node(s2t3) at ($(s2)!.6!(t3)$) {};\draw[->](s2)--(s2t3);

\draw (1,-2) rectangle (7,0);\node  at (0.5,-1){$H$};
\node (vn-1) at (2,-1) [p] {};\node  at (1.5,-1){$v_{n-1}$};
\node (v1) at (5,-1) [p] {};\node  at (4.7,-1){$v_1$};
\node (v0) at (6,-1) [p] {};\node  at (5.7,-1){$z$};
\node  at (3,-1) [q] {};\node  at (3.5,-1) [q] {};\node  at (4,-1) [q] {};

\draw (1.5,-5) rectangle (2.5,-3);\node  at (2,-5.3){$M_{v_{n-1}}$};
\node (vxn-1) at (2,-3.8) [p] {};\node  at (2,-4.3){$v'_{n-1}$};
\draw (4.5,-5) rectangle (5.5,-3);\node  at (5,-5.3){$M_{v_1}$};
\node (vx1) at (5,-3.8) [p] {};\node  at (5,-4.3){$v'_{1}$};
\node  at (3,-4) [q] {};\node  at (3.5,-4) [q] {};\node  at (4,-4) [q] {};

\draw (9,-2) rectangle (10,0);\node  at (10.7,-1){$M_z$};
\node (vx0) at (9.3,-1) [p] {};\node  at (9.7,-1){$z'$};

\draw (v0)--(vx0);\node(vx0m) at ($(v0)!.7!(vx0)$) {};\draw[->](v0)--(vx0m);
\draw (v1)--(vx1);\node(vx1m) at ($(v1)!.7!(vx1)$) {};\draw[->](v1)--(vx1m);
\draw (vn-1)--(vxn-1);\node(vxn-1m) at ($(vn-1)!.7!(vxn-1)$) {};\draw[->](vn-1)--(vxn-1m);
\draw (s1)--(v0);\node(s1v0) at ($(s1)!.7!(v0)$) {};\draw[->](s1)--(s1v0);
\draw (s1)--(v1);\node(s1v1) at ($(s1)!.7!(v1)$) {};\draw[->](s1)--(s1v1);
\draw (s1)--(vn-1);\node(s1vn-1) at ($(s1)!.7!(vn-1)$) {};\draw[->](s1)--(s1vn-1);
\draw (v1)--(t2);\node(v1t2) at ($(v1)!.6!(t2)$) {};\draw[->](v1)--(v1t2);
\draw (vn-1)--(t2);\node(vn-1t2) at ($(vn-1)!.65!(t2)$) {};\draw[->](vn-1)--(vn-1t2);
\draw (s1)--(9.5,0);\node(s10) at ($(s1)!.7!(9.5,0)$) {};\draw[->](s1)--(s10);
\draw (9.5,0)--(t2);\node(0t2) at ($(9.5,0)!.6!(t2)$) {};\draw[->](9.5,0)--(0t2);
\draw (v0)--(6.5,2);\node(v02) at ($(v0)!.7!(6.5,2)$) {};\draw[->](v0)--(v02);
\draw (v0)--(11.5,2);\node(v03) at ($(v0)!.7!(11.5,2)$) {};\draw[->](v0)--(v03);

\end{tikzpicture}
\end{center}
\caption{The graph $f(H)$.}
\label{TGraph G}
\end{figure}
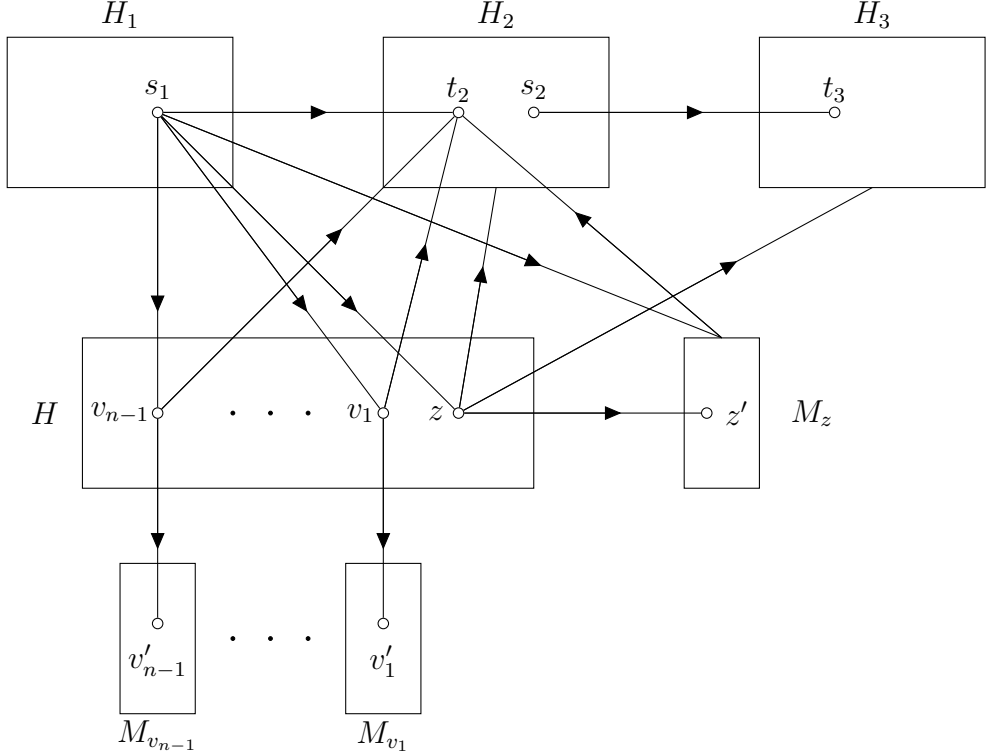

Figure~\ref{TGraph G} illustrates the test digraph $f(H)$ obtained from $H$.
If $H$ has $n$ vertices, then $H$ has $3n/2$ edges, and $f(H)$ has $9n+24$
vertices and $39n/2+78$ edges.  The test digraph can be produced in time
polynomial in the size of $H$.

For two disjoint subgraphs $D_1$ and $D_2$ of a (di)graph $D$, we denote by
$D_1+D_2$ the subgraph of $D$ induced by $V(D_1)\cup V(D_2)$.  It contains
$D_1$, $D_2$, and the arcs with endpoints in $V(D_1)$ and $V(D_2)$.

\begin{lemma}\label{HT}
If $H$ is a $3$-regular triangle-free Hamiltonian graph, then $b(f(H))=2$.
\end{lemma}

\begin{proof}
Let $G=f(H)$.  Since $G$ contains $\tK_{5,3}$ as a subgraph, $G$ is non-planar.
Hence Theorem~\ref{barvdig} yields $b(G)\geq 2$.  Our approach is to develop
special $2$-bar visibility representations for subgraphs of $G$ and combine
them when $H$ has a spanning cycle to obtain a $2$-bar visibility
representation of $G$.

{\bf Step 1.} Construct displayed $2$-bar visibility representations of $H_1$,
$H_2$, and $H_3$ so that in $H_1$ vertex $s_1$ is assigned a rightmost bar, in
$H_2$ vertices $s_2$ and $t_2$ are assigned a rightmost bar and a leftmost bar,
respectively, and in $H_3$ vertex $t_3$ is assigned a leftmost bar.  These
representations are guaranteed by Lemma \ref{Kvisibility}.
Combine the representations of $H_1$, $H_2$, and $H_3$ as shown in
Figure~\ref{staircase}.  This incorporates the arcs $s_1t_2$ and $s_2t_3$
without using any extra bar and also does not introduce any unwanted arc.

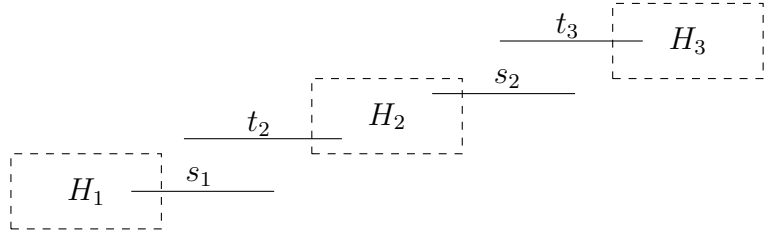
\begin{figure}[htbp]
\begin{center}
\begin{tikzpicture} [p/.style={circle,draw=black,fill=black,inner sep=0.5pt}]
\draw [dashed] (0,0) rectangle (2,1);
\draw [dashed] (4,1) rectangle (6,2);\draw [dashed] (8,2) rectangle (10,3);
\node  at (1,0.5){$H_1$};\node  at (5,1.5){$H_2$};\node  at (9,2.5){$H_3$};
\draw [-](1.6,0.5)--(3.5,0.5);
\draw [-](4.4,1.2)--(2.3,1.2);
\draw [-](5.6,1.8)--(7.5,1.8);
\draw [-](6.5,2.5)--(8.4,2.5);
\node  at (2.5,0.7){$s_1$};\node  at (3.3,1.4){$t_2$};
\node  at (6.6,2.0){$s_2$};\node  at (7.4,2.7){$t_3$};
\end{tikzpicture}
\caption{Representation of $H_1+H_2+H_3$.}
\label{staircase}
\end{center}
\end{figure}

{\bf Step 2.} Since $H$ is Hamiltonian, it decomposes into a spanning cycle $C$
and a perfect matching.  Construct a displayed representation of $C$ using one
bar for each vertex other than $z$ and two bars for $z$ (the leftmost and
rightmost bars).  Lemma~\ref{Kvisibility} provides a displayed $2$-bar
visibility representation of $M_z$ with $z'$ assigned the rightmost bar.
Place this representation on the left of the representation of $C$ to
incorporate the arc $zz'$ as shown in Figure \ref{HC}.

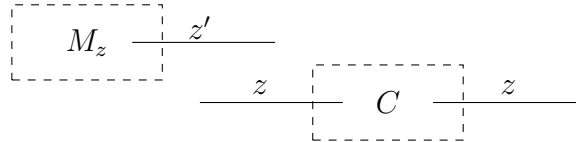
\begin{figure}[htbp]
\begin{center}
\begin{tikzpicture} [p/.style={circle,draw=black,fill=black,inner sep=0.5pt}]
\draw [dashed] (0,-0.7) rectangle (2,.3);
\draw [dashed] (4,-0.5) rectangle (6,-1.5);
\node  at (1,-0.2){$M_z$};\node  at (5,-1){$C$};
\draw [-](1.6,-0.2)--(3.5,-0.2);\draw [-](4.4,-1)--(2.5,-1);\draw [-](5.6,-1)--(7.5,-1);
\node  at (2.5,0.0){$z'$};\node  at (3.3,-0.8){$z$};\node  at (6.6,-0.8){$z$};
\end{tikzpicture}
\caption{Representation of $C+M_z$.}
\label{HC}
\end{center}
\end{figure}

{\bf Step 3.} Place the displayed representation of $C+M_z$ between the bars
for $s_1$ and $t_2$ to incorporate all arcs from $s_1$ to vertices in $C$ and
$M_z$, arcs from vertices in $C$ and $M_z$ to $t_2$.  By extending the
rightmost $z$-bar, we can represent all arcs from $z$ to vertices in $H_2$ and
$H_3$. See Figure \ref{H123CM}.

\begin{figure}[htbp]
\begin{center}
\begin{tikzpicture} [p/.style={circle,draw=black,fill=black,inner sep=0.5pt}]
\draw [dashed] (-2,-1.5) rectangle (0,-0.5);
\draw [dashed] (4,1.3) rectangle (6,2.3);
\draw [dashed] (8,2.4) rectangle (10,3.4);
\node  at (-1,-1){$H_1$};\node  at (5,1.8){$H_2$};\node  at (9,2.9){$H_3$};
\draw [-](-0.4,-1)--(3.5,-1);\draw [-](4.4,1.5)--(0.5,1.5);
\draw [-](5.6,2.1)--(7.5,2.1);\draw [-](6.5,2.9)--(8.4,2.9);
\node  at (1.5,-0.8){$s_1$};\node  at (2.3,1.7){$t_2$};
\node  at (6.6,2.3){$s_2$};\node  at (7.4,3.1){$t_3$};
\draw [dashed] (1,-0.2) rectangle (3,.8); \node  at (1.9,0.3){$C+M_z$};
\draw [-](2.8,0.3)--(11.5,0.3);\node  at (7.4,0.5){$z$};
\end{tikzpicture}
\caption{Representation of $H_1+H_2+H_3+C+M_z$.}
\label{H123CM}
\end{center}
\end{figure}
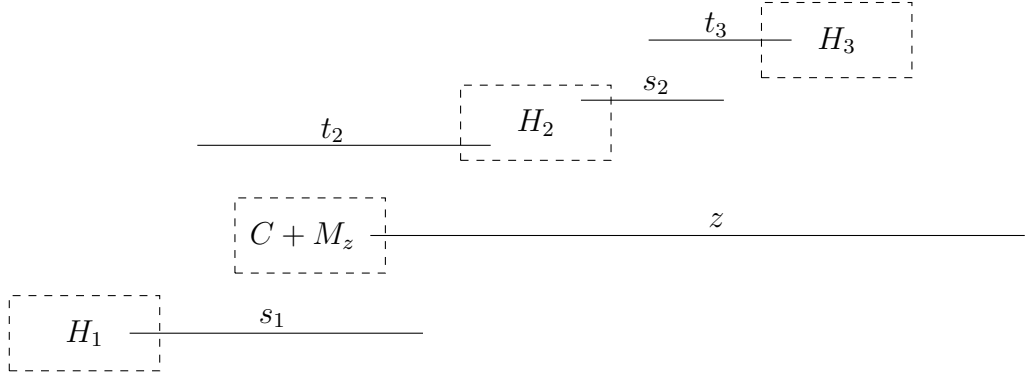

{\bf Step 4.} For each arc $uw$ in the perfect matching $E(H)-E(C)$
with $z\notin\{u,w\}$, construct  a displayed $2$-bar visibility representation
of $M_u+M_w+uw$ using two bars for each vertex in $M_u$ and $M_w$, one bar
for $u$, and one bar for $w$, as shown in Figure~\ref{Mu}.  If $z\in\{u,w\}$,
say $z=w$ by symmetry, then construct a displayed $2$-bar visibility
representation of $M_u+u$ and place it on the right side of the representation
of $H_1+H_2+H_3+C+M_z$ (Figure~\ref{H123CM}) to incorporate the arc $uz$
without introducing an extra bar for $z$.

Steps 1--4 complete a $2$-bar visibility representation of $G$.
\end{proof}

\begin{figure}[htbp]
\begin{center}
\begin{tikzpicture} [p/.style={circle,draw=black,fill=black,inner sep=0.5pt}]
\draw [dashed] (0,0) rectangle (2,1);
\draw [dashed] (7,0.2) rectangle (9,-0.8);
\node  at (1,0.5){$M_w$};\node  at (8,-0.3){$M_u$};
\draw [-](1.6,0.5)--(3.6,0.5);\draw [-](7.4,-0.3)--(5.4,-0.3);
\node  at (2.5,0.7){$w'$};\node  at (6.4,-0.1){$u'$};

\draw [-](2.6,-0.3)--(4.6,-0.3);\draw [-](6,-1.2)--(4,-1.2);
\node  at (3.5,-0.1){$w$};\node  at (5,-1){$u$};
\end{tikzpicture}
\caption{Representation of $M_u+M_w+uw$.}
\label{Mu}
\end{center}
\end{figure}
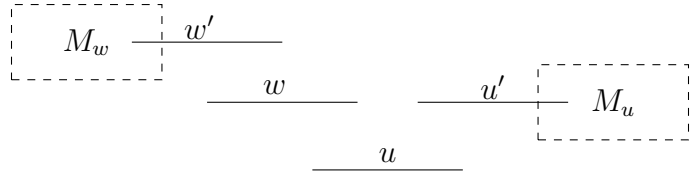

\begin{definition}\label{entire}
A bar $A$ {\it covers} a bar $B$ (above or below) if the projection of $B$ on
the horizontal axis is contained in the projection of $A$ on the axis.
\end{definition}

\begin{lemma}\label{TH}
Let $H$ be a $3$-regular triangle-free graph.
If $b(f(H))=2$, then $H$ is Hamiltonian.
\end{lemma}

\begin{proof}
Let $G=f(H)$, and let $\Psi$ be a $2$-bar visibility representation of $G$.
For $F\esub G$, let $\Psi(F)$ denote the set of bars representing $V(F)$ in
$\Psi$.  By Lemma \ref{Kvisibility}, each $H_i$ is $2$-bar tight, and in any
$2$-bar visibility representation of $G$ the bars for vertices of $H_i$ appear
contiguously.  By left-right symmetry, we may assume that bars for $H_1$
are to the left of those for $H_2$ and $H_3$.  Since there is only one arc
$s_it_{i+1}$ joining $H_i$ and $H_{i+1}$, avoiding unwanted visibilities
requires this arc to be represented by a visibility between the rightmost bar
in $\Psi(H_i)$ and the leftmost bar in $\Psi(H_{i+1})$.  Now $H_1+H_2+H_3$ must
be represented in $\Psi$ as shown in Figure~\ref{staircase}.  Note that
$\Psi(H_1+H_2+H_3)$ is contiguous, and the bars establishing $s_1t_2$ and
$s_2t_3$ are inner bars (not leftmost or rightmost bars) in $\Psi(H_1+H_2+H_3)$.

For each $v\in V(H)-\{z\}$, in order to avoid unwanted visibilities, we need
two bars for $v$: one bar $\overline{v}$ between bars for $s_1$ and $t_2$ in
$\Psi(H_1+H_2+H_3)$ to establish arcs $s_1v$ and $vt_2$, and one bar
$\underline{v}$ outside the horizontal extent of $\Psi(H_1+H_2+H_3)$ to
establish the arc $vv'$ (because $v'$ is not incident to any vertex in $H_i$).
Two claims about the bars that can be seen by $\overline{v}$ and
$\underline{v}$ will enable us to extract a spanning cycle in $H$ using the
bars of the form $\overline{v}$.

\smallskip
\noindent {\bf Claim 1:} {\it If $v\in V(H)-\{z\}$, then $\underline{v}$ sees a
bar for at most one vertex in $H$.}

Suppose that $\underline{v}$ sees a bar $\underline{p}$ for vertex
$p\in V(H)-\{z\}$.  Also $\underline{v}$ and $\underline{p}$ must see bars
for $v'$ and $p'$.  These must be end bars in $\Psi(M_v)$ and $\Psi(M_p)$,
since $v$ has no neighbor in $M_v$ other than $v'$, and similarly for $p$.
One end of $\underline{v}$ and one end of a bar for $v'$ establish $vv'$, and
similarly, one end of $\underline{p}$ and one end of a bar for $p'$ establish
$pp'$.  The other ends of $\underline{v}$ and $\underline{p}$ must be used to
represent $vp$ (or $pv$).  Thus $\underline{v}$ and $\underline{p}$ are inner
bars in $\Psi(M_v+M_p+vp)$ (or $\Psi(M(v)+M(p)+pv)$), which appears
contiguously (see Figure~\ref{Mu}).

Suppose that $\underline{v}$ also sees a bar for $q\in V(H)$.  If $q\ne z$,
then by the preceding paragraph $\underline{q}$ introduces an unwanted
visibility with something in $M_v+vv'$ or $M_p+pp'$.

Hence we may assume $q=z$.  Since $\underline{v}$ is an inner bar in
$\Psi(M_v+M_p+vp)$, and among those vertices $z$ can only be adjacent to $v$ and
$p$ (but not both, since $H$ has no triangle), the bar for $z$ seen by
$\underline{v}$ must be covered by $\underline{v}$ and see only $\underline{v}$.
Now the other bar $\hat z$ for $z$ must see bars for $z'$ and for two other
neighbors of $z$ in $H$.  Since $H$ is triangle-free, the inside bar
$\overline{w}$ for some vertex $w\in V(H)$ must now be covered by $\hat z$.
This obstructs the visibility between $\overline{w}$ and the bar for one of
$s_1$ or $t_2$, preventing one of $s_1w$ and $wt_2$ from being established.
The contradiction completes the proof of Claim 1.

\smallskip
\noindent {\bf Claim 2:} {\it If $v\in V(H)-\{z\}$, then $\overline{v}$ sees
bars for at most two vertices in $H$.}

Note that $\overline{v}$ cannot see both bars for a vertex $w\in H-\{z\}$,
since only one bar for $w$ is in $\Psi(H_1+H_2+H_3)$.  If $\overline{v}$ sees
both bars for $z$, then $\overline{v}$ and the two bars for $z$ must occur as
shown in Figure~\ref{vz}. The left bar for $z$ must see one end of
$\overline{v}$, because we use this bar to see $z'\in M_z$, and no other vertex
in $M_z$ is adjacent to $z$.  The right bar for $z$ must see the other end of
$\overline{v}$, because we use this $z$-bar to incorporate the arcs from $z$ to
$V(H_2)\cup V(H_3)$.  Since $H$ is $3$-regular, and by Claim 1 $\underline{v}$
establishes at most one arc involving a neighbor of $v$ besides $z$, at least
one neighbor $x$ of $v$ in $H$ still needs $\overline{x}$ to establish an arc
involving $x$ and $v$.  Because $H$ is triangle-free, $xz\notin E(H)$, so
$\overline{x}$ will be blocked by $\overline{v}$ from establishing the arc
$s_1x$ or the arc $xt_2$.  Hence $\overline{v}$ cannot see both bars for one
other vertex of $H$.

\begin{figure}[htbp]
\begin{center}
\begin{tikzpicture}
\draw [-](7.4,-0.3)--(5.4,-0.3);
\node  at (6.4,-0.1){$z$};
\draw [-](2.6,-0.3)--(4.6,-0.3);\draw [-](6,.5)--(4,.5);
\node  at (3.5,-0.1){$z$};\node  at (5,0.7){$v$};
\end{tikzpicture}
\caption{Representation of $zv$.}
\label{vz}
\end{center}
\end{figure}
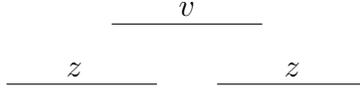

If $\overline{v}$ sees bars for three distinct vertices in $H$, then one of
them (say $\overline{y}$) must be covered by $\overline{v}$ (because $H$ is
triangle-free), which prevents $\overline{y}$ from establishing one of
$\{s_1y,yt_2\}$.

\smallskip
{\bf Conclusion:}
Since $H$ is $3$-regular, the consequence of Claims 1 and 2 is that for
$v\in V(H)-\{z\}$, the inside bar $\overline{v}$ sees bars for exactly two
neighbors of $v$ in $H$, and the outside bar $\underline{v}$ sees a bar for
exactly one neighbor of $v$ in $H$.  We can therefore follow the inside
bars from left to right as a path, with both ends of this path of bars being
bars for $z$.  This produces a Hamiltonian cycle in $H$.
\end{proof}

\begin{lemma}[\cite{West84}]\label{3HC}
Determining whether a $3$-regular triangle-free graph contains a Hamilton
cycle is NP-complete.
\end{lemma}

\begin{theorem}\label{2NPC}
Determining whether a digraph has the bar visibility number $2$ is NP-complete.
\end{theorem}

\begin{proof} From Lemmas \ref{HT} and \ref{TH}, $b(f(H))=2$ if and only if $H$
is Hamiltonian.  The claim then follows from Lemma \ref{3HC}.
\end{proof}

Furthermore, testing $b(G)\le t$ for digraphs is NP-complete for any fixed $t$
with $t\ge2$.  By constructing a digraph $\widetilde{G}$ from $G$ such that
$\widetilde{G}$ has a $t$-bar visibility representation if and only if $G$ has
a $(t-1)$-bar visibility representation, one can reduce $(t-1)$-bar visibility
representation to $t$-bar visibility representation.  The claim then follows
from Theorem \ref{2NPC}.  Let $G$ be an arbitrary digraph, and let
$\widetilde{G}$ be a digraph whose underlying graph is obtained from $G$ by
adding three copies of $K_{t^2+t-1,t+1}$ for each vertex $v\in V(G)$, with
one edge joining $v$ to the central copy and one edge joining each of the
two other copies to the central copy, their endpoints in the central copy
being adjacent (see Figure 5 and Theorem 2 of \cite{West84}).  In
$\widetilde{G}$, edges in each copy of $K_{t^2+t-1,t+1}$ are oriented from one
part to the other, and other edges are oriented arbitrarily.

\end{document}